\documentclass[12pt,reqno]{amsart}

\usepackage[utf8]{inputenc}
\usepackage[english]{babel}
\usepackage{csquotes}

\usepackage{geometry}
\usepackage{setspace}
\usepackage[bottom]{footmisc}

\usepackage{amsmath, amsfonts, amssymb, amsthm, mathtools}
\usepackage{dsfont}

\usepackage{graphicx}
\usepackage{caption}
\usepackage{float}
\usepackage{booktabs}
\usepackage{dcolumn}
\usepackage{pdflscape}
\usepackage{rotating}
\usepackage{xcolor}
\usepackage{titlesec}

\usepackage{fancyhdr}
\fancypagestyle{plain}{
  \fancyhf{}
  \fancyfoot[C]{\thepage}
  
}

\usepackage[
    colorlinks=true,
    linktoc=all,
    linkcolor=blue,
    citecolor=blue,
    urlcolor=blue
]{hyperref}

\numberwithin{equation}{section}

\newtheorem{theorem}{Theorem}[section]
\newtheorem{lemma}[theorem]{Lemma}

\newtheorem{corollary}[theorem]{Corollary}

\newtheorem{remark}[theorem]{Remark}
\newtheorem{assumption}[theorem]{Assumption}

\newcommand{\Top}{\mathcal{T}}
\newcommand{\Tmat}{\mathbf{T}}

\titleformat{\section}
  {\normalfont\Large\bfseries\color{black}}
  {\thesection}{1em}{}

\titleformat{\subsection}
  {\normalfont\large\bfseries\color{black}}
  {\thesubsection}{1em}{}

\titleformat{\subsubsection}
  {\normalfont\normalsize\bfseries\color{black}}
  {\thesubsubsection}{1em}{}

\begin{document}

\pagenumbering{gobble}
\begin{titlepage}
\thispagestyle{empty}

\begin{center}

\vspace*{2.5cm}

{\color{black}\LARGE\bfseries Spectral Characterisation of Covariance Existence in Markov-Switching Affine Recurrences \par}

\vspace{1.5cm}

{\Large Wilson Tai\textsuperscript{\dag}\par}

\vspace{0.5cm}

{\large Dublin City University\par}
{\large DCU School of Mathematical Sciences\par}

\vspace{0.75cm}

{\color{black}\Large\bfseries Abstract\par}

\vspace{0.3cm}

\begin{minipage}{1\textwidth}
\onehalfspacing
\raggedbottom
This paper provides a finite-dimensional spectral characterisation of covariance existence for finite-state Markov-switching affine recurrences of the form
\[
Z_{n+1}=A_{\Theta_{n+1}}Z_n+\zeta_{n+1}.
\]
While strict stationarity is typically governed by negativity of the top Lyapunov exponent of the random matrix products, this does not ensure finite variances, covariances, or Pearson correlations. We show that these second-order properties are governed instead by a Markov-switching Kronecker operator. Writing \(T_{ji}:=p_{ij}(A_j\otimes A_j)\) and denoting its spectral radius by \(\rho(T)\), we prove that, under a natural Perron-excitation condition on the innovation covariance, \(\rho(T)<1\) is necessary and sufficient for the stationary solution to be square-integrable. When this condition holds, the regime-weighted second moments solve a finite-dimensional linear system, yielding the covariance matrix explicitly. Thus, the mean-square spectral condition becomes an exact covariance-existence criterion for the stationary law. Examples illustrate the separation: the top Lyapunov exponent may be negative while the second-order spectral radius exceeds one, so strict stationarity can persist even when finite second moments, and hence population covariances and Pearson correlations, do not exist.
\end{minipage}

\vfill

\begin{minipage}{1\textwidth}
\footnotesize

\noindent\textsuperscript{\dag}\textbf{Correspondence:} Wilson Tai, \texttt{wilson.tai2@mail.dcu.ie}

\vspace{0.1cm}


\vspace{0.1cm}


\end{minipage}

\end{center}
\end{titlepage}


\setcounter{page}{1}
\pagenumbering{arabic}
\pagestyle{fancy}

\section{Introduction}

Markov-switching affine recurrences are often studied first through the question of stationarity: under what conditions does the process admit a well-defined stationary causal solution? In applied work, one often wants to go beyond stationarity and study second-order properties such as variances, covariances and correlations to characterise volatility and dependence. While these second-order quantities form the backbone of much applied analysis, their existence does not follow from strict stationarity alone. For example, in regime-switching models of asset returns or risk factors, portfolio, hedging and risk calculations commonly rely on well-defined population second moments \cite{JiangLiuTang2024, KoleVanDijk2023}. A strictly stationary Markov-switching affine recurrence may still have infinite second moments. In this paper, we ask when the stationary solution is square-integrable and provide a finite-dimensional spectral characterisation of precisely when its covariance matrix exists.

We consider finite-state Markov-switching affine recurrences of the form
\begin{equation}
Z_{n+1}=A_{\Theta_{n+1}}Z_n+\zeta_{n+1},
\qquad Z_n\in\mathbb R^d,
\label{eq:intro-recurrence}
\end{equation}
where \(A_j\in\mathbb R^{d\times d}\) for \(j\in S=\{1,\ldots,M\}\), and \((\Theta_n)\) is a finite-state Markov chain with transition matrix \(P=(p_{ij})\). The qualitative behaviour of such recurrences is usually governed by the almost-sure growth rate of the products \(A_{\Theta_n}A_{\Theta_{n-1}}\cdots A_{\Theta_1}\),
or equivalently by negativity of the associated top Lyapunov exponent, under standard logarithmic moment assumptions. This is the familiar pathwise contraction condition from the theory of stochastic recurrence equations, perpetuities, stochastic difference equations, and iterated random functions \cite{Vervaat1979,Brandt1986,DiaconisFreedman1999,BougerolPicard1992AnnProb,BougerolPicard1992JEconometrics}.

Our starting point is a recurrence which already admits a strictly stationary solution. We ask the more precise stationary-law question: under what additional condition is this solution square-integrable, so that its stationary variances, covariances, and Pearson correlations are well-defined? The answer is not determined by the top Lyapunov exponent. A recurrence may contract along typical sample paths while making sufficiently persistent visits to locally expansive regimes. Such episodes may be too rare to destroy strict stationarity, but persistent enough to destroy second moments.

The crucial ingredient is the Markov-switching Kronecker operator. Define the block matrix \(T\) by
\begin{equation}
T_{ji}:=p_{ij}(A_j\otimes A_j),
\qquad 1\leq i,j\leq M.
\label{eq:intro-Kronecker}
\end{equation}
This operator describes the propagation of regime-weighted second moments. We write \(\rho(T)\) for its spectral radius. Our main result is that, under a natural Perron-excitation condition on the innovation covariances,
\begin{equation}
\mathbb E\|Z\|^2<\infty
\qquad\Longleftrightarrow\qquad
\rho(T)<1.
\end{equation}
When \(\rho(T)<1\), the stationary covariance matrix is obtained from a finite-dimensional linear system. When \(\rho(T)\geq1\), strict stationarity may nevertheless persist, while the covariance and Pearson correlation matrices fail to exist. Thus, pathwise contraction alone is not sufficient to justify covariance-based calculations.

We contribute to the literature by turning this second-order spectral condition into an exact stationary-law criterion for covariance existence. Under Perron excitation, the condition is not merely sufficient but necessary for square-integrability. The role of Perron excitation is to exclude unstable second-order directions that are present in the Kronecker operator but never reached by the innovation covariance. The distinction is therefore simple: the top Lyapunov exponent governs existence of the stationary causal solution, while the Kronecker spectral radius governs whether that solution has finite second moments. The novelty here is that, under Perron excitation, the familiar second-order spectral condition becomes an exact necessary and sufficient criterion for the stationary law.

Our paper is related to a broad literature on stability and moment properties in switching and multiplicative-noise models, including regime-switching autoregressions and VARs \cite{Hamilton1989,Krolzig1997,BasakLu2005}, bilinear and nonlinear time-series models \cite{GrangerAndersen1978,LingPengZhu2015}, and systems with Markovian switching \cite{ShaoXi2015,YinZhu2010}. Recent work continues to study second-order structure in switching models: Cavicchioli \cite{Cavicchioli2022} analyses covariance and spectral properties of Markov-switching VARs, Kole and van Dijk \cite{KoleVanDijk2023} develop moment analysis for Markov-switching VARs, and Bibi and Hamdi \cite{BibiHamdi2024} study \(L^2\) existence and covariance structure in Markov-switching bilinear processes. Closer to our paper is Basak and Lu \cite{BasakLu2005}, who study stationarity and second-order properties of Markov-switching VAR-type models. While these works derive or use second-order properties within particular model classes, our focus is an exact covariance-existence criterion for the broader affine recurrence.

Our contribution is also distinct from the Kesten--Goldie literature \cite{Kesten1973,Goldie1991,BuraczewskiDamekMikosch2016}, which concerns tail behaviour and regular variation, and from the random-coefficient and ARCH/GARCH literature \cite{NichollsQuinn1982,Engle1982,Bollerslev1986,Nelson1990,BougerolPicard1992JEconometrics}, where stationarity and moment existence are known to be separate restrictions. Finally, \(\rho(T)<1\) is familiar from mean-square stability theory for Markov jump linear systems \cite{CostaFragosoMarques2005}. Note that the spectral condition itself is not new in that literature. What differs here is its role: rather than studying decay of second moments in a homogeneous dynamic system, we use the same spectral object to characterise existence of finite second moments for the stationary law of the affine recurrence.

More precisely, if \(M_j:=\mathbb E[ZZ^\top 1_{\{\Theta=j\}}]\), then, whenever the equivalent second-order conditions hold,
\[
M_j=\sum_{i=1}^M p_{ij}A_jM_iA_j^\top+\pi_jQ_j,
\qquad j=1,\ldots,M,
\]
and the unconditional covariance matrix is \(\Sigma:=\sum_{j=1}^M M_j\).

The remainder of the paper is organised as follows. Section~\ref{sec:prelim} introduces the model, the stationary solution, and the second-moment operator. Section~\ref{sec:main} establishes the main spectral criterion and the covariance reconstruction formula. Section~\ref{sec:examples} gives examples which demonstrate that strict stationarity and finite covariance are genuinely different properties. Section~\ref{sec:conclusion} concludes.
\section{Preliminaries}
\label{sec:prelim}
This section fixes the notation and introduces the two spectral objects used throughout the paper: the top Lyapunov exponent and the Markov-switching Kronecker operator.

\subsection{Markov-switching affine recurrence}

Let \(S=\{1,\ldots,M\}\), and let \((\Theta_n)_{n\in\mathbb Z}\) be a stationary, irreducible and aperiodic Markov chain on \(S\) with transition matrix
\begin{equation}
P=(p_{ij})_{1\leq i,j\leq M},
\qquad
p_{ij}
=
\mathbb P(\Theta_{n+1}=j\mid \Theta_n=i).
\label{eq:transition-matrix}
\end{equation}
Let \(\pi=(\pi_1,\ldots,\pi_M)\) denote its invariant distribution:
\begin{equation}
\pi_j
=
\sum_{i=1}^M\pi_i p_{ij},
\qquad
\sum_{j=1}^M\pi_j=1.
\label{eq:invariant-distribution}
\end{equation}
Since the chain is finite and irreducible, \(\pi_j>0\) for every \(j\in S\); see, for example, \cite{MeynTweedie2009}.

We study the Markov-switching affine recurrence
\begin{equation}
Z_{n+1}=A_{\Theta_{n+1}}Z_n+\zeta_{n+1},
\qquad Z_n\in\mathbb R^d,
\label{eq:prelim-recurrence}
\end{equation}
where \(A_j\in\mathbb R^{d\times d}\) for \(j\in S\). This is a finite-state Markov-dependent analogue of the stochastic recurrence equation
\[
X=AX+B,
\]
studied, for example, in \cite{Brandt1986,BougerolPicard1992AnnProb,BuraczewskiDamekMikosch2016}.

For \(n\geq 1\), define the multiplicative product \(\Phi_n := A_{\Theta_n}A_{\Theta_{n-1}}\cdots A_{\Theta_1}\). The top Lyapunov exponent is
\begin{equation}
\gamma
=
\lim_{n\to\infty}
\frac{1}{n}\log\|\Phi_n\|,
\label{eq:top-lyapunov}
\end{equation}
whenever the limit exists. Under the standing assumptions below, this limit exists almost surely and is constant.

We write \(\mathbb S^d\) for the space of real symmetric \(d\times d\) matrices, and \(\mathbb S^d_+\) for its positive-semidefinite cone.

\subsection{Standing assumptions}

We impose the following assumptions throughout.

\begin{assumption}[Stationarity]
\label{ass:stationarity}
The top Lyapunov exponent is negative:
\begin{equation}
\gamma<0.
\label{eq:negative-lyapunov}
\end{equation}
Moreover, \(\mathbb E\log^+\|\zeta_0\|<\infty\).
\end{assumption}

Under Assumption~\ref{ass:stationarity}, the recurrence \eqref{eq:prelim-recurrence} admits a unique stationary causal solution, given by the backward series
\begin{equation}
Z
=
\sum_{\ell=0}^{\infty}
A_{\Theta_0}A_{\Theta_{-1}}\cdots A_{\Theta_{-\ell+1}}\zeta_{-\ell},
\label{eq:stationary-backward-representation}
\end{equation}
where the \(\ell=0\) product is the identity. This is the standard contractive case for affine stochastic recurrences \cite{Brandt1986,BougerolPicard1992AnnProb,BuraczewskiDamekMikosch2016}.

\begin{assumption}[Innovations]
\label{ass:innovations}
Let \(\mathcal A^{\Theta}:=\sigma(\Theta_k:k\in\mathbb Z)\) be the sigma-algebra generated by the regime path. Conditional on
\(\mathcal A^{\Theta}\), the innovations \((\zeta_n)_{n\in\mathbb Z}\) are independent across time. Moreover, for every \(n\),
\begin{equation}
\mathbb E[\zeta_{n+1}\mid \mathcal A^{\Theta}]=0,
\label{eq:conditional-mean-zero}
\end{equation}
and
\begin{equation}
\mathbb E\!\left[
\zeta_{n+1}\zeta_{n+1}^{\top}
\mid
\mathcal A^{\Theta}
\right]
=
Q_{\Theta_{n+1}},
\qquad
Q_j\in\mathbb S^d_+,\quad \operatorname{tr}(Q_j)<\infty,\quad j\in S.
\label{eq:conditional-second-moment}
\end{equation}
\end{assumption}

Assumption~\ref{ass:innovations} permits the innovation law to depend on the realised regime path through the contemporaneous state \(\Theta_{n+1}\). It is not an unconditional independence assumption between innovations and regimes. Its purpose is more specific: conditional on the regime path, the new innovation is centred and independent of past innovations. This is the condition that allows the second-moment dynamics to close.

For \(N\geq 0\), define the finite truncation of the stationary backward series by
\[
Z_n^{(N)}
:=
\sum_{\ell=0}^{N}
A_{\Theta_n}A_{\Theta_{n-1}}\cdots A_{\Theta_{n-\ell+1}}
\zeta_{n-\ell},
\]
where the \(\ell=0\) product is the identity. Then \(Z_n^{(N)}\) is measurable with respect to the regime path and finitely many past innovations. Hence, conditional on \(\mathcal A^{\Theta}\), it is independent of \(\zeta_{n+1}\). Therefore, for each \(j\in S\),
\begin{equation}
\mathbb E\!\left[
Z_n^{(N)}\zeta_{n+1}^{\top}1_{\{\Theta_{n+1}=j\}}
\right]
=
0,
\qquad
\mathbb E\!\left[
\zeta_{n+1}(Z_n^{(N)})^{\top}1_{\{\Theta_{n+1}=j\}}
\right]
=
0.
\label{eq:truncated-cross-orthogonality}
\end{equation}
When the stationary solution is already known to be square-integrable, the same identities hold with \(Z_n\) in place of \(Z_n^{(N)}\):
\begin{equation}
\mathbb E\!\left[
Z_n\zeta_{n+1}^{\top}1_{\{\Theta_{n+1}=j\}}
\right]
=
0,
\qquad
\mathbb E\!\left[
\zeta_{n+1}Z_n^{\top}1_{\{\Theta_{n+1}=j\}}
\right]
=
0.
\label{eq:cross-term-orthogonality}
\end{equation}
Accordingly, all second-moment identities below are first applied either to finite truncations or under the explicit hypothesis \(Z\in L^2\). In this sense, no second-moment finiteness is used before it has been proved.

The conditional centering assumption is imposed only to avoid carrying first-moment terms through the covariance recursion. Nonzero conditional means could be handled by augmenting the system with the corresponding first-moment equations. By contrast, conditional serial correlation in the innovations would introduce additional lag-covariance terms, and the Markov-switching Kronecker operator would no longer close the second-moment dynamics by itself.

\subsection{Regime-weighted second moments}

We now introduce the auxiliary objects that allow efficient study of covariance existence. For a stationary solution \(Z\), define the regime-weighted second moments by
\begin{equation}
M_j
:=
\mathbb E[ZZ^\top\mathbf 1_{\{\Theta=j\}}],
\qquad j=1,\ldots,M,
\label{eq:regime-second-moment-prelim}
\end{equation}
whenever these matrices are finite. Since the recurrence has no deterministic drift and the innovations are conditionally centred, the stationary mean is zero whenever it exists. In particular, whenever the stationary solution is square-integrable,
\begin{equation}
\mathbb E[Z]=0.
\label{eq:stationary-mean-zero}
\end{equation}
Thus, whenever the second moments are finite, the unconditional covariance matrix is
\begin{equation}
\Sigma
=
\operatorname{Var}(Z)
=
\mathbb E[ZZ^\top]
=
\sum_{j=1}^M M_j.
\label{eq:covariance-from-regime-moments}
\end{equation}

The remaining question is therefore whether the regime-weighted second moments \(M_j\) are finite. We turn next to the operator that governs their propagation.

\subsection{The Markov-switching Kronecker operator}
\label{subsec:2.4}
The evolution of regime-weighted second moments is governed by a Markov-switching Kronecker operator. For \(H=(H_1,\ldots,H_M)\in(\mathbb S^d)^M\), define
\begin{equation}
(\Top H)_j
=
\sum_{i=1}^M p_{ij}A_jH_iA_j^\top,
\qquad j=1,\ldots,M.
\label{eq:second-order-operator-prelim}
\end{equation}
After vectorisation, \(\Top\) is represented by the block matrix \(\Tmat\) with blocks
\begin{equation}
\Tmat_{ji}
=
p_{ij}(A_j\otimes A_j),
\qquad 1\leq i,j\leq M.
\label{eq:block-T-prelim}
\end{equation}
Thus, \(\Tmat\in\mathbb R^{Md^2\times Md^2}\).
Indeed, when
\[
\operatorname{vec}(H)
=
\begin{pmatrix}
\operatorname{vec}(H_1)\\
\vdots\\
\operatorname{vec}(H_M)
\end{pmatrix},
\]
then
\begin{equation}
\operatorname{vec}(\Top H)
=
\Tmat\operatorname{vec}(H).
\label{eq:operator-vectorisation}
\end{equation}
We write
\begin{equation}
\rho(\Top)=\rho(\Tmat).
\label{eq:spectral-radius-T}
\end{equation}
Throughout, \(\rho(\mathcal T)\) denotes the spectral radius of the induced linear operator on
\((\mathbb S^d)^M\). Equivalently, it is the Perron root of the vectorised block representation restricted to the invariant subspace of vectorised symmetric matrices.

The connection with mean-square stability theory for Markov jump linear systems is useful here \cite{CostaFragosoMarques2005}. That literature typically studies the homogeneous system and asks whether second moments decay to zero from arbitrary initial conditions. Our question is different. We start from an affine recurrence that already possesses a strictly stationary solution under \(\gamma<0\), and ask whether its stationary law has finite second moments. The same Kronecker operator appears in both settings because it governs the propagation of second moments. Here, however, innovation covariance is injected at every date, so covariance existence requires these propagated contributions to be summable. Under Perron excitation, this gives \(\rho(\Top)<1\) a stationary-law interpretation: it becomes the exact criterion for square-integrability of the stationary solution.

Finally, for $j \in \{1,...,M\}$, define \(q_j:=\pi_jQ_j\).

\subsection{Perron excitation}

For the necessity direction, we need to rule out a simple degeneracy: the second-order operator may possess an unstable positive direction that is never reached by the innovation covariance.

\begin{assumption}[Perron excitation]
\label{ass:perron-excitation}
Let \(R=(R_1,\ldots,R_M)\neq0\), with \(R_j\succeq0\) for every \(j\in S\), be any dual Perron eigenvector of \(\Top\), so that
\[
\Top^\ast R=\rho(\Top)R.
\]
Then
\[
\sum_{j=1}^M \operatorname{tr}(R_jq_j)>0.
\]
\end{assumption}

Assumption~\ref{ass:perron-excitation} is a non-degeneracy condition, not a stability condition. It requires the innovation covariance to place positive variation in every positive Perron direction associated with the spectral radius. The point is simple. If \(\rho(\Top)\geq1\) but the innovations are completely orthogonal to the corresponding unstable direction, that direction is present in the operator but irrelevant for the stationary covariance generated by the affine recursion. After all, an unstable direction that is never excited cannot make the stationary second moments diverge. Perron excitation rules out precisely this case, which ensures that an unstable Perron direction of \(\Top\) is also an active direction of the second-moment dynamics.

\begin{lemma}[A simple sufficient condition]
\label{lem:sufficient-excitation}
If \(Q_j\succ0\) for every \(j\in S\), then Assumption~\ref{ass:perron-excitation} holds.
\end{lemma}

\begin{proof}
Let \(R=(R_1,\ldots,R_M)\neq0\), with \(R_j\succeq0\) for every \(j\in S\), be a dual Perron eigenvector of \(\Top\). Then at least one \(R_j\) is nonzero. Since \(\pi_j>0\) for every \(j\), and \(Q_j\succ0\), we have
\[
\operatorname{tr}(R_jq_j)
=
\pi_j\operatorname{tr}(R_jQ_j)>0
\]
for at least one \(j\). Therefore
\[
\sum_{j=1}^M \operatorname{tr}(R_jq_j)>0.
\]
\end{proof}
\section{Main result}
\label{sec:main}

We now prove the main covariance-existence criterion. The main theorem has two components. The first is probabilistic: when \(\rho(\Top)<1\), the stationary second moments are obtained by summing the propagated innovation covariances. The second is deterministic: if a finite nonnegative solution to the moment equation exists, then Perron excitation rules out \(\rho(\Top)\geq1\). The first barrier is therefore to separate the probabilistic construction of the second moments from the deterministic obstruction imposed by Perron directions. We isolate this obstruction first.

\begin{lemma}[Perron obstruction to finite second moments]
\label{lem:perron-obstruction}
Define
\[
\mathcal K
=
\{H=(H_1,\ldots,H_M)\in(\mathbb S^d)^M:H_j\succeq0,\ j=1,\ldots,M\}
=
(\mathbb S_+^d)^M.
\]
Suppose that \(\Top\mathcal K\subseteq\mathcal K\). Let \(q=(q_1,\ldots,q_M)\in\mathcal K\), and suppose that \(M\in\mathcal K\) is finite and satisfies
\begin{equation}
M=\Top M+q.
\label{eq:abstract-moment-equation}
\end{equation}
If Assumption~\ref{ass:perron-excitation} holds, then \(\rho(\Top)<1\).
\end{lemma}

\begin{proof}
Since \(\Top\) is a positive linear operator on the finite-dimensional cone \(\mathcal K\), the finite-dimensional Perron--Frobenius theorem implies that the dual operator \(\Top^*\) admits a nonzero element \(R=(R_1,\ldots,R_M)\in\mathcal K\) associated with the spectral radius. Thus
\[
\Top^*R=\rho(\Top)R.
\]
Equivalently, under the usual column-wise vectorisation, the block tuple \(R\in(\mathbb S^d)^M\) has stacked representation
\[
r:=\operatorname{vec}(R)
=
\begin{pmatrix}
\operatorname{vec}(R_1)\\
\vdots\\
\operatorname{vec}(R_M)
\end{pmatrix}
\in\mathbb R^{Md^2\times1}.
\]
We use the dual pairing
\[
\langle R,H\rangle
=
\sum_{j=1}^M\operatorname{tr}(R_jH_j),
\]
which is in line with the standard vector inner product,
\(
\langle R,H\rangle
=
r^\top\operatorname{vec}(H).
\)

Pairing \eqref{eq:abstract-moment-equation} with \(R\) gives
\[
\langle R,M\rangle
=
\langle R,\Top M\rangle+\langle R,q\rangle.
\]
By definition of the adjoint and the Perron relation,
\[
\langle R,\Top M\rangle
=
\langle \Top^*R,M\rangle
=
\rho(\Top)\langle R,M\rangle.
\]
Therefore
\begin{equation}
(1-\rho(\Top))\langle R,M\rangle
=
\langle R,q\rangle.
\label{eq:perron-obstruction-identity}
\end{equation}

Now \(R_j\succeq0\) and \(M_j\succeq0\), so \(\langle R,M\rangle\geq0\). Moreover, by Assumption~\ref{ass:perron-excitation},
\[
\langle R,q\rangle
=
\sum_{j=1}^M\operatorname{tr}(R_jq_j)>0.
\]
If \(\rho(\Top)\geq1\), then the left-hand side of \eqref{eq:perron-obstruction-identity} is nonpositive, while the right-hand side is strictly positive. This is impossible, which thus gives us \(\rho(\Top)<1\).
\end{proof}

In essence, the lemma says that, provided the second-moment equation has a finite positive-semidefinite solution, Perron excitation forces the second-order operator to be stable. With this tool in hand, we can now apply the argument to the Markov-switching affine recurrence.

\begin{theorem}[Covariance-existence criterion]
\label{thm:second-order-spectral-criterion}
Suppose Assumptions~\ref{ass:stationarity}, \ref{ass:innovations} and \ref{ass:perron-excitation} hold.

The stationary solution is square-integrable if and only if \(\rho(\Top)<1\), that is,
\[
\mathbb E\|Z\|^2<\infty
\quad\Longleftrightarrow\quad
\rho(\Top)<1.
\]

When these equivalent conditions hold, the regime-weighted second moments \(M_1,\ldots,M_M\) are the unique finite solution of
\begin{equation}
M_j
=
\sum_{i=1}^M p_{ij}A_jM_iA_j^\top+\pi_jQ_j,
\qquad j=1,\ldots,M.
\label{eq:regime-moment-equation}
\end{equation}
In vectorised form,
\begin{equation}
\operatorname{vec}(M)
=
(I-\Tmat)^{-1}\operatorname{vec}(q).
\label{eq:second-moment-solution}
\end{equation}

The following second-order quantities are then well-defined:

\textbf{(i)} The covariance matrix is
\begin{equation*}
\Sigma
=
\operatorname{Var}(Z)
=
\mathbb E[ZZ^\top]
=
\sum_{j=1}^M M_j.
\label{eq:covariance-matrix}
\end{equation*}

\textbf{(ii)} For components \(Z^a\) and \(Z^b\),
\begin{equation*}
\operatorname{Var}(Z^a)=\Sigma_{aa},
\qquad
\operatorname{Cov}(Z^a,Z^b)=\Sigma_{ab}.
\label{eq:variance-covariance-entries}
\end{equation*}

\textbf{(iii)} If \(\Sigma_{aa}>0\) and \(\Sigma_{bb}>0\), then
\begin{equation*}
\operatorname{Corr}(Z^a,Z^b)
=
\frac{\Sigma_{ab}}{\sqrt{\Sigma_{aa}\Sigma_{bb}}}.
\label{eq:correlation-entry}
\end{equation*}
\end{theorem}

\begin{proof}
\textbf{Part 1: Sufficiency.}

Suppose \(\rho(\Top)<1\). By Assumption~\ref{ass:stationarity}, the recurrence admits the stationary backward representation
\[
Z
:=
\sum_{\ell=0}^{\infty}
A_{\Theta_0}A_{\Theta_{-1}}\cdots A_{\Theta_{-\ell+1}}\zeta_{-\ell},
\]
where the \(\ell=0\) product is the identity.

Recall the finite truncation
\[
Z^{(N)}
:=
\sum_{\ell=0}^{N}
A_{\Theta_0}A_{\Theta_{-1}}\cdots A_{\Theta_{-\ell+1}}\zeta_{-\ell},
\qquad N\geq0.
\]
Let
\[
M_j^{(N)}
=
\mathbb E\!\left[
Z^{(N)}(Z^{(N)})^\top
\mathbf 1_{\{\Theta_0=j\}}
\right],
\qquad j=1,\ldots,M,
\]
and write
\[
M^{(N)}
:=
\bigl(M_1^{(N)},\ldots,M_M^{(N)}\bigr).
\]
Also set
\[
B_\ell
:=
A_{\Theta_0}A_{\Theta_{-1}}\cdots A_{\Theta_{-\ell+1}},
\qquad B_0=I.
\]

Since each \(B_\ell\) is measurable with respect to \(\mathcal A^\Theta\), the cross terms can be checked conditionally on the realised regime path. For \(\ell\neq m\), \(B_\ell\), \(B_m\), and \(\mathbf 1_{\{\Theta_0=j\}}\) are \(\mathcal A^\Theta\)-measurable. Hence, by conditional independence and conditional centering,
\[
\begin{aligned}
\mathbb E\!\left[
B_\ell\zeta_{-\ell}\zeta_{-m}^{\top}B_m^\top
\mathbf 1_{\{\Theta_0=j\}}
\right]
&=
\mathbb E\!\left[
B_\ell
\mathbb E\!\left[
\zeta_{-\ell}\zeta_{-m}^{\top}
\mid \mathcal A^\Theta
\right]
B_m^\top
\mathbf 1_{\{\Theta_0=j\}}
\right] \\
&=0.
\end{aligned}
\]

As a consequence, only the diagonal lag contributions remain.

The zeroth lag contributes \(q_j=\pi_jQ_j\) in regime \(j\), yielding the \(M\)-tuple \(q=(q_1,\ldots,q_M)\). Moving one step forward through the Markov chain applies the second-order transition operator \(\Top\): if a regime-weighted second moment \(H_i\) is located in regime \(i\) at time \(n\), then, after one transition to regime \(j\), its contribution at time \(n+1\) is
\[
H_i\mapsto p_{ij}A_jH_iA_j^\top.
\]
Consequently, a contribution \(H\) propagated forward by \(\ell\) steps becomes \(\Top^\ell H\). Starting from \(q\), the lag-\(\ell\) contribution is therefore \(\Top^\ell q\). Thus
\begin{equation}
M^{(N)}
=
\sum_{\ell=0}^{N}\Top^\ell q.
\label{eq:truncated-second-moment-series}
\end{equation}

Since \(\rho(\Top)<1\), the Neumann series \(\sum_{\ell=0}^{\infty}\Top^\ell\) converges on the finite-dimensional space of \(M\)-tuples of symmetric matrices. Therefore, the right-hand side of \eqref{eq:truncated-second-moment-series} converges to a finite \(M\)-tuple of positive-semidefinite matrices. In particular,
\[
\sup_{N\geq0}
\sum_{j=1}^M
\operatorname{tr}\bigl(M_j^{(N)}\bigr)
<\infty.
\]
Since
\[
\mathbb E\|Z^{(N)}\|^2
=
\sum_{j=1}^M
\mathbb E\!\left[
\|Z^{(N)}\|^2\mathbf 1_{\{\Theta_0=j\}}
\right]
=
\sum_{j=1}^M
\operatorname{tr}\bigl(M_j^{(N)}\bigr),
\]
the sequence \(\{Z^{(N)}\}\) is uniformly bounded in \(L^2\). Since \(Z^{(N)}\to Z\) almost surely by the stationary backward representation, Fatou's lemma gives
\[
\mathbb E\|Z\|^2
\leq
\liminf_{N\to\infty}
\mathbb E\|Z^{(N)}\|^2
<\infty.
\]

Therefore, the stationary solution is square-integrable.

We next derive the fixed-point equation for the second moments. Work in stationarity and fix \(j\in S\). On the event \(\{\Theta_{n+1}=j\}\),
\[
Z_{n+1}=A_jZ_n+\zeta_{n+1}.
\]
Thus
\[
M_j
=
\mathbb E\!\left[
Z_{n+1}Z_{n+1}^\top
\mathbf 1_{\{\Theta_{n+1}=j\}}
\right].
\]
Expanding \(Z_{n+1}Z_{n+1}^\top\) gives four terms:
\[
A_jZ_nZ_n^\top A_j^\top,
\qquad
A_jZ_n\zeta_{n+1}^\top,
\qquad
\zeta_{n+1}Z_n^\top A_j^\top,
\qquad
\zeta_{n+1}\zeta_{n+1}^\top.
\]

The two cross terms vanish after multiplication by \(\mathbf 1_{\{\Theta_{n+1}=j\}}\). Indeed, conditional on \(\mathcal A^\Theta\), the random vector \(Z_n\) is generated by past innovations and the regime path, while \(\zeta_{n+1}\) is conditionally independent of past innovations and conditionally centred. Hence,
\[
\mathbb E\!\left[
A_jZ_n\zeta_{n+1}^\top
\mathbf 1_{\{\Theta_{n+1}=j\}}
\right]
=
0,
\qquad
\mathbb E\!\left[
\zeta_{n+1}Z_n^\top A_j^\top
\mathbf 1_{\{\Theta_{n+1}=j\}}
\right]
=
0.
\]

Therefore,
\[
M_j
=
\mathbb E\!\left[
A_jZ_nZ_n^\top A_j^\top
\mathbf 1_{\{\Theta_{n+1}=j\}}
\right]
+
\mathbb E\!\left[
\zeta_{n+1}\zeta_{n+1}^\top
\mathbf 1_{\{\Theta_{n+1}=j\}}
\right].
\]

For the first term, conditioning on \(\Theta_n=i\) gives
\[
\mathbb E\!\left[
A_jZ_nZ_n^\top A_j^\top
\mathbf 1_{\{\Theta_{n+1}=j\}}
\right]
=
\sum_{i=1}^M
p_{ij}A_jM_iA_j^\top.
\]
For the innovation term, Assumption~\ref{ass:innovations} gives
\[
\begin{aligned}
\mathbb E\!\left[
\zeta_{n+1}\zeta_{n+1}^\top
\mathbf 1_{\{\Theta_{n+1}=j\}}
\right]
&=
\mathbb E\!\left[
\mathbb E\!\left[
\zeta_{n+1}\zeta_{n+1}^\top
\mid \mathcal A^\Theta
\right]
\mathbf 1_{\{\Theta_{n+1}=j\}}
\right] \\
&=
\mathbb E\!\left[
Q_{\Theta_{n+1}}
\mathbf 1_{\{\Theta_{n+1}=j\}}
\right] \\
&=
\pi_jQ_j.
\end{aligned}
\]
Therefore,
\[
M_j
=
\sum_{i=1}^M p_{ij}A_jM_iA_j^\top+\pi_jQ_j,
\]
which proves \eqref{eq:regime-moment-equation}. In operator form,
\[
M=\Top M+q.
\]
Since \(\rho(\Top)<1\), the operator \(I-\Top\) is invertible, and hence
\[
M=(I-\Top)^{-1}q.
\]

Vectorising the fixed-point equation gives
\[
\operatorname{vec}(M)
=
\Tmat\operatorname{vec}(M)
+
\operatorname{vec}(q).
\]
Since
\[
\rho(\Tmat)=\rho(\Top)<1,
\]
the matrix \(I-\Tmat\) is invertible. Therefore,
\[
\operatorname{vec}(M)
=
(I-\Tmat)^{-1}\operatorname{vec}(q),
\]
which proves \eqref{eq:second-moment-solution}. Uniqueness follows from the invertibility of \(I-\Top\), equivalently of \(I-\Tmat\).

\textbf{Part 2: Necessity.}

Suppose that \(\mathbb E\|Z\|^2<\infty\). Then each regime-weighted second moment
\[
M_j
=
\mathbb E\!\left[
Z_nZ_n^\top\mathbf 1_{\{\Theta_n=j\}}
\right]
\]
is finite. The preceding second-moment calculation is therefore justified and gives
\[
M=\Top M+q.
\]
Moreover, \(M_j\succeq0\) and \(q_j=\pi_jQ_j\succeq0\) for every \(j\). Thus \(M\) is a finite solution in the cone
\(
\mathcal K=(\mathbb S_+^d)^M.
\)
By Lemma~\ref{lem:perron-obstruction}, we obtain \(\rho(\Top)<1\). This proves necessity.

Combining sufficiency and necessity proves the equivalence.

Finally, since the stationary solution is centred by \eqref{eq:stationary-mean-zero},
\[
\operatorname{Var}(Z)=\mathbb E[ZZ^\top].
\]
Using the regime decomposition,
\[
\mathbb E[ZZ^\top]
=
\sum_{j=1}^M
\mathbb E\!\left[
ZZ^\top\mathbf 1_{\{\Theta=j\}}
\right]
=
\sum_{j=1}^M M_j.
\]
This establishes the stated covariance formula. The variance, covariance and correlation identities then follow immediately from the entries of \(\Sigma\). This completes the proof.
\end{proof}

The spectral-radius condition \(\rho(\Top)<1\) is familiar from mean-square stability theory for Markov jump linear systems. The point here is its stationary-law interpretation: for the affine recurrence, the innovation covariance is repeatedly propagated through the tensorised Markov-switching dynamics, so the stationary covariance exists precisely when this propagation is summable. Perron excitation is the non-degeneracy condition that ensures any non-summable Perron direction is actually reached by the innovations.

\begin{remark}[Relation to the existing literature]
Second-order properties of Markov-switching autoregressive models are studied, among others, by Basak and Lu \cite{BasakLu2005}. Theorem~\ref{thm:second-order-spectral-criterion} gives the corresponding stationary-law result for the affine recurrence, with Perron excitation providing the non-degeneracy condition required for necessity. The connection with mean-square stability for Markov jump linear systems is discussed in Section~\ref{subsec:2.4}. The paper is also related to the Kesten--Goldie theory and its extensions \cite{Kesten1973,Goldie1991,BuraczewskiDamekMikosch2016}, which concern tail asymptotics and regular variation, whereas our focus is the finite-versus-infinite second-moment question.
\end{remark}

\begin{corollary}[Scalar reduction]
\label{cor:scalar-reduction}
Suppose \(d=1\), so that
\[
Z_{n+1}=a_{\Theta_{n+1}}Z_n+\zeta_{n+1}.
\]
For ease of notation, let \(D_2:=\operatorname{diag}(a_1^2,\ldots,a_M^2)\). If the scalar innovation variances satisfy \(Q_j>0\) for every \(j\in S\), then
\begin{equation}
\mathbb E Z^2<\infty
\quad\Longleftrightarrow\quad
\rho(PD_2)<1.
\label{eq:scalar-second-moment-condition}
\end{equation}
\end{corollary}

\begin{proof}
In the scalar case, \(A_j=a_j\). Hence, the vectorised second-order matrix \(\Tmat\) has entries \(\Tmat_{ji}=p_{ij}a_j^2\). We therefore obtain
\[
\Tmat=(PD_2)^\top.
\]
Since spectral radius is invariant under transposition,
\[
\rho(\Tmat)=\rho(PD_2).
\]
The claim follows from Theorem~\ref{thm:second-order-spectral-criterion}. The condition \(Q_j>0\) for all \(j\) guarantees Perron excitation by Lemma~\ref{lem:sufficient-excitation}.
\end{proof}
\section{Examples}
\label{sec:examples}
We now provide two examples illustrating the separation between strict stationarity and covariance existence. For ease of exposition, we first consider a scalar example in which both the Lyapunov condition and the second-order spectral condition can be verified analytically. The second is a two-dimensional numerical illustration, where the same covariance-threshold mechanism appears in a genuinely multivariate setting and variances and covariances interact through non-diagonal regime matrices.

\subsection{A scalar example}

Consider the scalar Markov-switching affine recurrence of \eqref{eq:intro-recurrence}, where \((\Theta_n)\) is a two-state Markov chain with transition matrix
\begin{equation}
P=
\begin{pmatrix}
49/50 & 1/50\\
1/5 & 4/5
\end{pmatrix},
\label{eq:scalar-example-P}
\end{equation}
and \(a_1=1/2\), \(a_2=3/2\). The invariant distribution is
\[
\pi_1=\frac{1/5}{1/50+1/5}=\frac{10}{11},
\qquad
\pi_2=\frac{1/50}{1/50+1/5}=\frac{1}{11}.
\]
Hence, by the ergodic theorem for finite-state Markov chains, the top Lyapunov exponent is
\begin{equation}
\gamma
=
\frac{10}{11}\log\frac12
+
\frac{1}{11}\log\frac32
=
\frac{1}{11}\log\left(\frac{3}{2048}\right)
<0.
\label{eq:scalar-example-lyapunov}
\end{equation}
As a result, the recurrence is contractive on average and admits a strictly stationary causal solution under the standing logarithmic moment assumptions on the innovations.

The important difference appears when one asks whether this stationary solution has a finite second moment. In the scalar case, the second-order spectral criterion reduces to
\begin{equation}
\mathbb E Z^2<\infty
\quad\Longleftrightarrow\quad
\rho(PD_2)<1,
\qquad
D_2=\operatorname{diag}(a_1^2,a_2^2).
\label{eq:scalar-example-second-order-criterion}
\end{equation}
Here
\begin{equation}
PD_2
=
\begin{pmatrix}
49/50 & 1/50\\
1/5 & 4/5
\end{pmatrix}
\begin{pmatrix}
1/4 & 0\\
0 & 9/4
\end{pmatrix}
=
\begin{pmatrix}
49/200 & 9/200\\
1/20 & 9/5
\end{pmatrix}.
\label{eq:scalar-example-PD2}
\end{equation}
Let \(p(\lambda)=\det(\lambda I-PD_2)\). Then
\[
p(1)
=
\det(I-PD_2)
=
\det
\begin{pmatrix}
151/200 & -9/200\\
-1/20 & -4/5
\end{pmatrix}
=
-\frac{97}{160}<0.
\]
Since \(PD_2\) is a positive matrix, its Perron root is the largest positive eigenvalue. The inequality \(p(1)<0\) implies that this Perron root is larger than one, so \(\rho(PD_2)>1\). Consequently, if the innovation variance is positive in each regime, the stationary solution exists but is not square-integrable:
\[
\mathbb E Z^2=\infty.
\]

This example gives an explicit analytic separation:
\[
\gamma<0
\qquad\text{but}\qquad
\rho(PD_2)>1.
\]
Thus, strict stationarity does not necessarily imply finite second moments.

\subsection{A two-dimensional numerical covariance-threshold illustration}

This example is deliberately numerical. The reported Lyapunov exponents are computed by normalised product iteration along a long simulated path of the finite-state Markov chain, while the second-order spectral radii and covariance matrices are obtained directly from the associated finite-dimensional Markov-switching Kronecker operator. The aim is to illustrate the covariance threshold in a genuinely multivariate setting, rather than to provide a fully analytic calculation of the Lyapunov exponent.

Again, consider the Markov-switching affine recurrence of \eqref{eq:intro-recurrence}, but this time \((\Theta_n)\) has transition matrix
\begin{equation}
P=
\begin{pmatrix}
0.98 & 0.02\\
0.20 & 0.80
\end{pmatrix}.
\label{eq:twod-example-P}
\end{equation}
The invariant distribution is
\[
\pi_1=\frac{0.20}{0.02+0.20}\approx 0.9091,
\qquad
\pi_2=\frac{0.02}{0.02+0.20}\approx 0.0909.
\]
Let the innovation covariance matrix be the same in both regimes:
\begin{equation}
Q_1=Q_2=Q,
\qquad
Q=
\begin{pmatrix}
1 & 0.4\\
0.4 & 1.2
\end{pmatrix}.
\label{eq:twod-example-Q}
\end{equation}
Since \(Q\) is positive definite, the Perron-excitation condition is satisfied.

\subsubsection{Strict stationarity with finite covariance: \texorpdfstring{\(\gamma<0\) and \(\rho(T)<1\)}{gamma < 0 and rho(T) < 1}}
\label{subsubsec:1}

First take
\begin{equation}
A_1=
\begin{pmatrix}
0.55 & 0.12\\
-0.08 & 0.48
\end{pmatrix},
\qquad
A_2=
\begin{pmatrix}
0.98 & 0.28\\
0.11 & 0.85
\end{pmatrix}.
\label{eq:twod-example-A-finite}
\end{equation}
Regime \(1\) is locally contractive, whereas regime \(2\) is locally expansive, with eigenvalues approximately \(1.1021\) and \(0.7279\). So regime \(2\) can amplify some directions of the state vector. What matters, however, is that the Markov chain spends most of its time in the contractive regime. Estimating the top Lyapunov exponent by simulating a long path of the Markov chain and applying the standard normalised product iteration to \(A_{\Theta_n}A_{\Theta_{n-1}}\cdots A_{\Theta_1}\) yields \(\gamma\approx -0.589<0\). In this sense, the recurrence is contractive on average and admits a strictly stationary causal solution.\footnote{The numerical computations and figures below were carried out in Python. The code is available upon request.}

The associated second-order operator is represented by the block matrix \(T\) with blocks
\[T_{ji}=p_{ij}(A_j\otimes A_j),
\qquad
i,j\in\{1,2\}.\]
For the parameter values in \eqref{eq:twod-example-A-finite}, direct computation of the associated second-order operator gives
\[
\rho(T)\approx 0.9734<1.
\]
Therefore, by Theorem~\ref{thm:second-order-spectral-criterion}, the stationary solution is square-integrable. The regime-weighted second moments \(M_1\) and \(M_2\) solve
\[M_j=\sum_{i=1}^2 p_{ij}A_jM_iA_j^\top+\pi_jQ,
\qquad
j=1,2.\]
Solving this finite-dimensional linear system gives
\[M_1\approx
\begin{pmatrix}
2.1344 & 0.6171\\
0.6171 & 1.4092
\end{pmatrix},
\qquad
M_2\approx
\begin{pmatrix}
6.7541 & 3.0232\\
3.0232 & 1.5401
\end{pmatrix}.\]
Hence, the unconditional covariance matrix is
\begin{equation}
\Sigma=M_1+M_2
\approx
\begin{pmatrix}
8.8885 & 3.6403\\
3.6403 & 2.9493
\end{pmatrix}.
\label{eq:twod-example-Sigma}
\end{equation}
Consequently,
\[
\operatorname{Var}(Z^1)\approx 8.8885,
\qquad
\operatorname{Var}(Z^2)\approx 2.9493,
\qquad
\operatorname{Cov}(Z^1,Z^2)\approx 3.6403.
\]
The Pearson correlation is therefore
\[
\operatorname{Corr}(Z^1,Z^2)
=
\frac{3.6403}{\sqrt{8.8885\cdot 2.9493}}
\approx 0.7110.
\]
Thus, in this parameter regime, the stationary solution has a well-defined covariance and correlation matrix. Notice that the covariance matrix is not a scalar multiple of the innovation covariance matrix. The off-diagonal entries are generated jointly by the innovation covariance and by the non-diagonal regime matrices.

\subsubsection{Strict stationarity without finite covariance: \texorpdfstring{\(\gamma<0\) and \(\rho(T)>1\)}{gamma < 0 and rho(T) > 1}}
\label{subsubsec:2}
Now keep the same transition matrix \eqref{eq:twod-example-P}, the same contractive regime \(A_1\), and the same innovation covariance matrix \eqref{eq:twod-example-Q}, but replace the second regime by
\begin{equation}
A_2=
\begin{pmatrix}
1.02 & 0.30\\
0.12 & 0.88
\end{pmatrix}.
\label{eq:twod-example-A-infinite}
\end{equation}
The eigenvalues of this matrix are approximately \(1.1522\) and \(0.7478\). Thus, regime \(2\) is again locally expansive. Despite this, because the chain still spends most of its time in regime \(1\), the same normalised product iteration gives \(\gamma\approx -0.586<0\). Hence, the recurrence remains strictly stationary.

The second-order picture is different. The spectral radius is now
\[
\rho(T)\approx 1.0637>1.
\]
By Theorem~\ref{thm:second-order-spectral-criterion},
\[
\mathbb E\|Z\|^2=\infty.
\]
As a result, the stationary solution exists, but it has no finite covariance matrix. Consequently, the Pearson correlation matrix is not well-defined.

What this example makes clear is that the relevant threshold is second-order rather than pathwise. The top Lyapunov exponent governs contraction along typical sample paths, while the Markov-switching Kronecker operator governs the propagation of innovation covariances. A relatively small change in the locally expansive regime can therefore leave \(\gamma<0\) while pushing \(\rho(T)\) above one. The process remains strictly stationary, but its finite second moments cease to exist and Pearson correlations are no longer well-defined.

\pagebreak 

\subsubsection{Numerical Illustrations}
Although our focus is theoretical, a simple simulation helps develop a better understanding of the second-order threshold. We proceed as follows. 

We first keep the transition matrix \(P\) and innovation covariance \(Q\) fixed and consider the two specifications of \(A_2\) in Sections~\ref{subsubsec:1} and \ref{subsubsec:2}. For each specification, we simulate the Markov chain from its invariant distribution and generate innovations from \(Normal(0,Q)\). Starting from \(Z_0=0\), we then iterate the affine recurrence forward. To make the comparison as clean as possible, the two specifications are fed the same realised regime path and the same sequence of innovations. After discarding an initial burn-in period (i.e., the first observations, to reduce the influence of the initial condition), we compute the empirical variances, covariance, and Pearson correlation recursively using the first \(N\) observations. Put differently, across the two simulations, the only change is the regime matrix \(A_2\), and hence whether the second-order spectral radius lies below or above one. Figure~\ref{fig:sample_moments} below shows the result.

\begin{figure}[H]
    \centering
    \includegraphics[width=18cm]
    {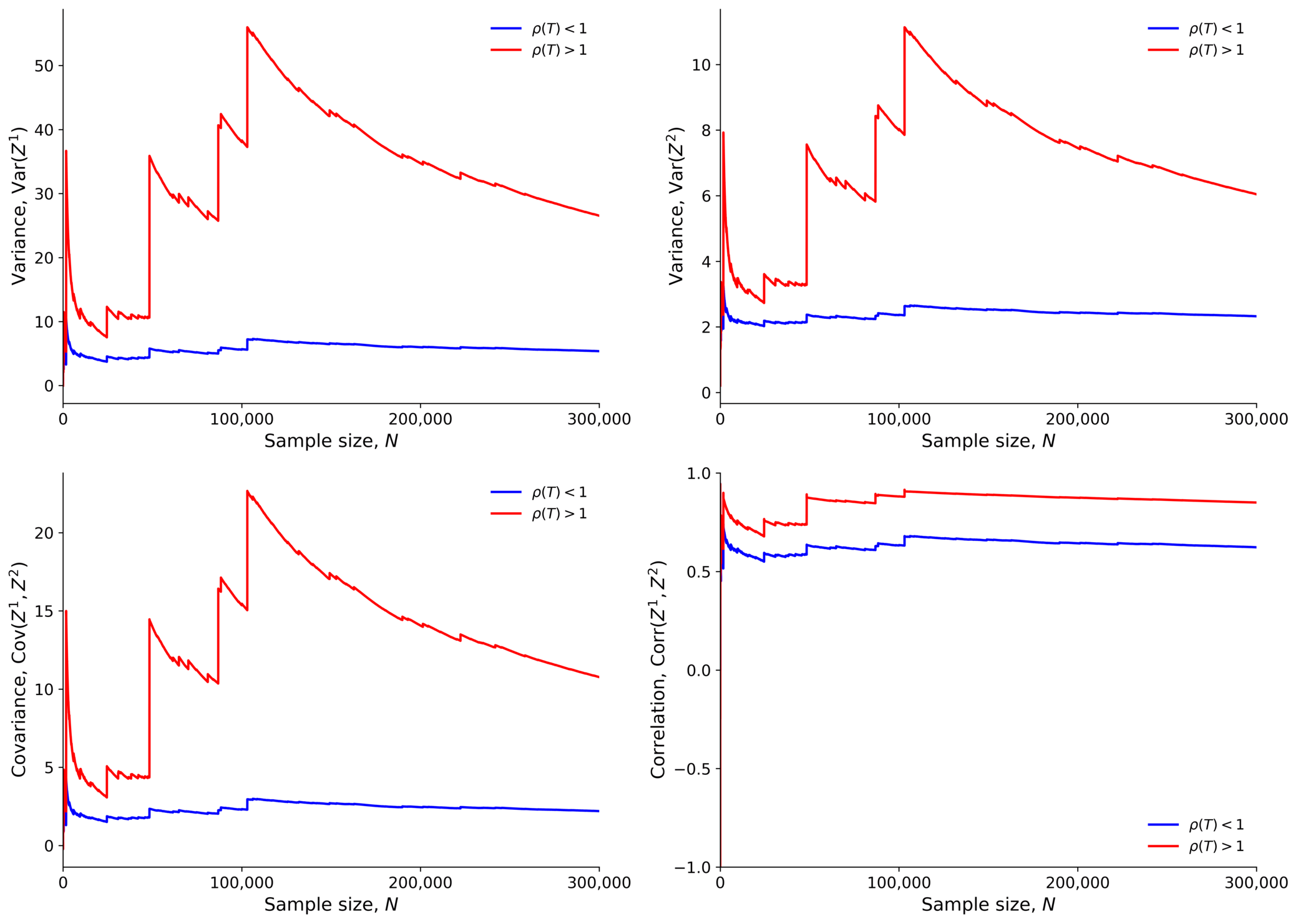}
    \caption{Sample moments across the second-order spectral threshold.}
    \label{fig:sample_moments}

    \vspace{0.2cm}
    \begin{minipage}{16cm}
        \footnotesize
        \textit{Note:} The blue lines correspond to the specification in Section~\ref{subsubsec:1}, with \(\rho(T)\approx0.9734<1\), and the red lines to the specification in Section~\ref{subsubsec:2}, with \(\rho(T)\approx1.0637>1\). Both specifications have a negative top Lyapunov exponent. The same simulated regime path and innovation sequence are used in both cases. Each panel reports the corresponding empirical moment computed recursively from the first \(N\) post-burn-in observations.
    \end{minipage}
\end{figure}

The blue lines correspond to the specification in Section~\ref{subsubsec:1}, where \(\rho(T)\approx0.9734<1\), while the red lines correspond to Section~\ref{subsubsec:2}, where \(\rho(T)\approx1.0637>1\). In both cases the top Lyapunov exponent remains negative, so the process is strictly stationary. The difference appears in the second-order objects. When \(\rho(T)<1\), the sample variances and covariance settle towards finite levels. When \(\rho(T)>1\), large observations repeatedly feed into the empirical second moments, producing sharp revisions even at large sample sizes. The process itself remains stationary, but its sample second moments do not settle around finite population counterparts. The sample correlation remains mechanically bounded between \(-1\) and \(1\), but in the latter case there is no finite-population Pearson correlation for it to estimate.

The implications for stability are strong. Consider, for example, a regime-switching model for asset returns or risk factors. An econometrician may observe a stable time series and reasonably conclude that the underlying process is stationary. Yet if \(\rho(T)\geq1\), the covariance matrix need not exist. Estimates of variances and correlations can then change sharply as the sample grows, simply because a few large observations receive substantial weight. This matters for familiar objects such as minimum-variance portfolios, hedge ratios, or covariance-based measures of risk, all of which rely on well-defined second-order quantities \cite{JiangLiuTang2024, KoleVanDijk2023}. These calculations may produce perfectly sensible numbers in any given sample, while having no finite population quantity to which they are converging.
\section{Conclusion}
\label{sec:conclusion}
We have shown that, for finite-state Markov-switching affine recurrences, strict stationarity and finite covariance are governed by different spectral objects. The contribution of this paper is not the mean-square stability inequality alone, but its stationary-law interpretation: conditional on strict stationarity, the Markov-switching Kronecker operator gives an exact criterion for whether the stationary distribution has finite covariance.

The top Lyapunov exponent controls existence of a stationary causal solution, while the Markov-switching Kronecker operator controls whether that solution is square-integrable. Thus, variances, covariances and Pearson correlations are legitimate stationary objects precisely in the second-order stable region \(\rho(T)<1\), subject to the Perron-excitation condition required for necessity.

The examples show that this distinction is more than technical. A recurrence may contract along typical sample paths while sufficiently persistent visits to locally expansive regimes destroy its second moments. The process remains stationary, but its covariance and correlation matrices are no longer well-defined. The practical implication is straightforward: covariance calculations in Markov-switching affine models require a second-order spectral check in addition to the usual stationarity condition.



\begin{thebibliography}{99}

\bibitem{BasakLu2005}
Basak, G.K. and Lu, Z.-Q. (2005) `Stationarity of switching VAR and other related models', \emph{arXiv preprint}, math/0507267.

\bibitem{BibiHamdi2024}
Bibi, A. and Hamdi, F. (2025) `Covariance analysis and GMM estimation of Markov switching bilinear processes', \emph{Statistics}, \textbf{59}(2), pp.~354--387.

\bibitem{Bollerslev1986}
Bollerslev, T. (1986) `Generalized autoregressive conditional heteroskedasticity', \emph{Journal of Econometrics}, \textbf{31}(3), pp.~307--327.

\bibitem{BougerolPicard1992AnnProb}
Bougerol, P. and Picard, N. (1992) `Strict stationarity of generalized autoregressive processes', \emph{The Annals of Probability}, \textbf{20}(4), pp.~1714--1730.

\bibitem{BougerolPicard1992JEconometrics}
Bougerol, P. and Picard, N. (1992) `Stationarity of GARCH processes and of some nonnegative time series', \emph{Journal of Econometrics}, \textbf{52}(1--2), pp.~115--127.

\bibitem{Brandt1986}
Brandt, A. (1986) `The stochastic equation \(Y_{n+1}=A_nY_n+B_n\) with stationary coefficients', \emph{Advances in Applied Probability}, \textbf{18}(1), pp.~211--220.

\bibitem{BuraczewskiDamekMikosch2016}
Buraczewski, D., Damek, E. and Mikosch, T. (2016) \emph{Stochastic Models with Power-Law Tails: The Equation \(X=AX+B\)}. New York: Springer.

\bibitem{Cavicchioli2022}
Cavicchioli, M. (2022) `Goodness-of-fit tests for Markov Switching VAR models using spectral analysis', \emph{Journal of Statistical Planning and Inference}, \textbf{219}, pp.~189--203.

\bibitem{CostaFragosoMarques2005}
Costa, O.L.V., Fragoso, M.D. and Marques, R.P. (2005) \emph{Discrete-Time Markov Jump Linear Systems}. London: Springer.

\bibitem{DiaconisFreedman1999}
Diaconis, P. and Freedman, D. (1999) `Iterated random functions', \emph{SIAM Review}, \textbf{41}(1), pp.~45--76.

\bibitem{Engle1982}
Engle, R.F. (1982) `Autoregressive conditional heteroscedasticity with estimates of the variance of United Kingdom inflation', \emph{Econometrica}, \textbf{50}(4), pp.~987--1007.

\bibitem{Goldie1991}
Goldie, C.M. (1991) `Implicit renewal theory and tails of solutions of random equations', \emph{The Annals of Applied Probability}, \textbf{1}(1), pp.~126--166.

\bibitem{GrangerAndersen1978}
Granger, C.W.J. and Andersen, A.P. (1978) \emph{An Introduction to Bilinear Time Series Models}. G{\"o}ttingen: Vandenhoeck \& Ruprecht.

\bibitem{Hamilton1989}
Hamilton, J.D. (1989) `A new approach to the economic analysis of nonstationary time series and the business cycle', \emph{Econometrica}, \textbf{57}(2), pp.~357--384.

\bibitem{JiangLiuTang2024}
Jiang, B., Liu, C. and Tang, C.Y. (2024) `Dynamic covariance matrix estimation and portfolio analysis with high-frequency data', \emph{Journal of Financial Econometrics}, \textbf{22}(2), pp.~461--491.

\bibitem{Kesten1973}
Kesten, H. (1973) `Random difference equations and renewal theory for products of random matrices', \emph{Acta Mathematica}, \textbf{131}, pp.~207--248.

\bibitem{KoleVanDijk2023}
Kole, E. and van Dijk, D. (2023) `Moments, shocks and spillovers in Markov-switching VAR models', \emph{Journal of Econometrics}, \textbf{236}(2), article~105474.

\bibitem{Krolzig1997}
Krolzig, H.-M. (1997) \emph{Markov-Switching Vector Autoregressions: Modelling, Statistical Inference, and Application to Business Cycle Analysis}. Lecture Notes in Economics and Mathematical Systems, vol.~454. Berlin: Springer.

\bibitem{LingPengZhu2015}
Ling, S., Peng, L. and Zhu, F. (2015) `Inference for a special bilinear time series model', \emph{Journal of Time Series Analysis}, \textbf{36}(1), pp.~61--76.

\bibitem{MeynTweedie2009}
Meyn, S.P. and Tweedie, R.L. (2009) \emph{Markov Chains and Stochastic Stability}. 2nd edn. Cambridge: Cambridge University Press.

\bibitem{Nelson1990}
Nelson, D.B. (1990) `Stationarity and persistence in the GARCH(1,1) model', \emph{Econometric Theory}, \textbf{6}(3), pp.~318--334.

\bibitem{NichollsQuinn1982}
Nicholls, D.F. and Quinn, B.G. (1982) \emph{Random Coefficient Autoregressive Models: An Introduction}. New York: Springer.

\bibitem{ShaoXi2015}
Shao, J. and Xi, F. (2014) `Stability and recurrence of regime-switching diffusion processes', \emph{SIAM Journal on Control and Optimization}, \textbf{52}(6), pp.~3496--3516.

\bibitem{Vervaat1979}
Vervaat, W. (1979) `On a stochastic difference equation and a representation of non-negative infinitely divisible random variables', \emph{Advances in Applied Probability}, \textbf{11}(4), pp.~750--783.

\bibitem{YinZhu2010}
Yin, G.G. and Zhu, C. (2010) \emph{Hybrid Switching Diffusions: Properties and Applications}. New York: Springer.

\end{thebibliography}
\end{document}